\documentclass[11pt,leqno]{amsart}
\usepackage{epsfig}
\usepackage{amssymb}
\usepackage{amscd}
\usepackage[matrix,arrow]{xy}

\newtheorem{theo}{Theorem}[section]
\newtheorem{lemma}[theo]{Lemma}
\newtheorem{defi}[theo]{Definition}
\newtheorem{prop}[theo]{Proposition}

\newtheorem{cor}[theo]{Corollary}
\newtheorem{remark}[theo]{Remark}

\numberwithin{equation}{section}

\def\bR{{\mathbf R}}
\def\bL{{\mathbf L}}

\def\pre-tr{\operatorname{pre-tr}}

\def\Hom{\operatorname{Hom}}

\newcommand{\bbC}{{\mathbb C}}

\newcommand{\bbZ}{{\mathbb Z}}

\newcommand{\bbP}{{\mathbb P}}

\newcommand{\cO}{{\mathcal O}}
\newcommand{\cP}{{\mathcal P}}

\newcommand{\cA}{{\mathcal A}}
\newcommand{\cB}{{\mathcal B}}

\newcommand{\cE}{{\mathcal E}}

\newcommand{\Perf}{\operatorname{Perf}}

\newcommand{\Eu}{\operatorname{Eu}}

\newcommand{\Tr}{\operatorname{Tr}}

\newcommand{\even}{\operatorname{even}}
\newcommand{\odd}{\operatorname{odd}}

\newcommand{\id}{\operatorname{id}}

\newcommand{\op}{\operatorname{op}}

\newcommand{\pt}{\operatorname{pt}}

\newcommand{\sh}{\operatorname{sh}}
\newcommand{\Hg}{\operatorname{Hg}}

\title[Lefschetz fixed point theorems for Fourier-Mukai functors and DG algebras]
{Lefschetz fixed point theorems for Fourier-Mukai functors and DG algebras}

\author{Valery A.~Lunts}
\address{Department of Mathematics, Indiana University,
Bloomington, IN 47405, USA} \email{vlunts@indiana.edu}

\thanks{The author was partially supported by the NSF grant 48-294-16}

\begin{document}

\begin{abstract} We propose some variants of Lefschetz fixed point
theorem for Fourier-Mukai functors on a smooth projective algebraic
variety. Independently we also suggest a similar theorem for
endo-functors on the category of perfect modules over a smooth
and proper DG algebra.
\end{abstract}

\maketitle

\tableofcontents

\section{Introduction}

Leschetz fixed point theorem (LFP) is a principle that has many
incarnations. One of its simplest forms is the following: Let
$f:X\to X$ be a nice self-map of a nice space $X.$ Then the
number of fixed points of $f$ equals the supertrace
$$\sum _i(-1)^i\Tr H_i(f)$$
where $H_i(f):H_i(X)\to H_i(X)$ is the induced map on homology.
The "number if fixed points" should be properly defined as the
intersection $\Gamma (f)\cdot \Delta $ of the graph of $f$ with the
diagonal $\Delta \subset X\times X.$

In this paper we prove several variants of LFP theorem.

In the first part we work in algebraic-geometric setting.
Let $X$ be a smooth
projective variety over a field  $k.$ Denote by $D^b(X)$ the bounded derived
category of coherent sheaves on $X.$
Let $Y$ be another smooth projective  variety over $k$ and
$E\in D^b(X\times Y).$ Then there is the corresponding
Fourier-Mukai functor $\Phi _E:D^b(X)\to D^b(Y)$
$$\Phi _E(-)=\bR p_*(E\stackrel {\bL}{\otimes }q^*(-)),$$
where the maps $p,q$ are the projections
$$X\stackrel{q}{\leftarrow} X\times Y\stackrel{p}{\to} Y.$$

One has a finite dimensional
graded vector space $HH_\bullet (X)$ - the
Hochschild homology of $X.$
The functor $\Phi _E$ induces the linear map of graded spaces
$HH_\bullet (\Phi _E):HH_\bullet (X)\to HH_\bullet (Y).$
In particular if $X=Y$ we get endomorphisms
$HH_i(\Phi _E):HH_i(X)\to HH_i(X)$
for each $i\in \bbZ.$
It is natural to define for $E\in D^b(X\times X)$
the "number of fixed points of $\Phi _E$" to be the "intersection" of
$E$ with the diagonal $\Delta \subset X\times X,$ i.e. as
$\sum _i(-1)^i\dim HH_i(E)$
(see Definition \ref{general-hochschild-homology}).
The following Hochschild homology version
of LFP theorem (= Theorem \ref{LFP-Hochschild-homology-geometric-in-text})
is easy to prove.

\begin{theo} \label{LFP-geometric}
In the above notation there is the equality
$$\sum _i(-1)^i\dim HH_i(E)=\sum _j(-1)^j\Tr HH_j(\Phi _E).$$
\end{theo}
The proof of this theorem has a tautological flavor
once basic functorial properties of $HH(X)$
have been established. Here the main references are \cite{Cal1},\cite{Ram},
\cite{MaSte}.

It is as easy to prove the Hirzebruch-Riemann-Roch theorem for Hochschild homology
(Proposition \ref{HRR-HH-geometric}).

Assume now that $k=\bbC.$ Then one can consider the singular cohomology
$H^\bullet (X,\bbC).$  Again an object $E\in D^b(X\times Y)$
induces the linear map $H^\bullet (\Phi _E):H^\bullet (X,\bbC)\to
H^\bullet (Y,\bbC)$ which is the convolution with the cohomology class
$ch(E)\cup \sqrt {td_{X\times Y}}\in H^\bullet (X\times Y,\bbC)$
(here $ch(E)$ is the
Chern character of $E$ and $td_{X\times Y}$ is the Todd class of
$X\times Y$). This map preserves the parity of the
degree of cohomology, hence it is the sum of two linear operators
$H^{\text{ev}} (\Phi _E)\oplus H^{\text{odd}} (\Phi _E).$
Next is the singular cohomology version of LFP theorem
(= Theorem \ref{LFP-singular-cohomology-in-text}).

\begin{theo} \label{LFP-singular-cohomology}
Let $X$ be a complex smooth projective variety and
let $E\in D^b(X\times X).$ Then there is the equality
$$\sum _i(-1)^i\dim HH_i(E)=\Tr H^{\text{ev}}(\Phi _E)-
\Tr H^{\text{odd}}(\Phi _E).$$
\end{theo}
This theorem follows from Theorem \ref{LFP-geometric} above
and Theorem 1.2 in
\cite{MaSte} (which in turn is heavily based on \cite{Cal1},\cite{Cal2}
and \cite{Ram}).

Denote by $\Delta :X\to X\times X$ the diagonal embedding. It follows from
the Grothendieck-Riemann-Roch theorem that
$$\sum (-1)^i\dim HH_i(E)
=\int _{X\times X}chE\cup \Delta _*(td _X)=
\int _X\Delta ^*(chE) \cup td _X$$
(Remark \ref{formula-for-dot-by-class}). So the left hand side in
Theorems \ref{LFP-geometric},\ref{LFP-singular-cohomology},\ref{LFP-singular-homology}
can be computed using the Chern character of
$E.$

Consider now the singular homology
$H_\bullet (X)=H_\bullet (X,\bbC ).$
Let $f:X\to X$ be a morphism.
For each $j$ we get the corresponding linear map $H_j(f):H_j(X)\to H_j(X).$
Again it is natural to define the "number of fixed points of $f$"
as the alternating sum
$\sum _i(-1)^iHH_i(\cO _{\Gamma (f)}),$
where $\cO _{\Gamma (f)}\in D^b(X\times X)$ is the structure sheaf of the
graph $\Gamma (f)$ of the morphism $f.$

\begin{theo} \label{LFP-singular-homology}
Let $X$ be a smooth complex projective variety and let
$f:X\to X$ be a morphism. Then in the above notation there is the equality
$$\sum _i(-1)^i\dim HH_i(\cO _{\Gamma (f)})=\sum _j(-1)^j\Tr H_j(f).$$
\end{theo}

This theorem (= Theorem \ref{LFP-singular-homology-in-text})
is a consequence of the special case of Theorem
\ref{LFP-singular-cohomology} (when $E=\cO _{\Gamma (f)}$) and
the Poincare duality between the singular homology and cohomology of $X.$
Theorem \ref{LFP-singular-homology} is not new: a similar formula
can be proved for any Weil (co)homology theory (of which singular cohomology
is an example) (see for instance \cite{Mus}). Nevertheless we consider it natural
to derive Theorem \ref{LFP-singular-homology} from Theorem \ref{LFP-geometric},
since in that last theorem both sides of the equality have the same nature -
Hochschild homology.

Finally in Section \ref{two-maps} we prove yet another version of LFP
theorem for two maps between different spaces (Theorem
\ref{lefschetz-theorem-for-two-maps}).

In the second part of the paper we want to consider a LFP theorem
of categorical nature:
the space $X$ is a triangulated category $T$ and the map $f$ is an
endofunctor $F:T\to T.$ More precisely, let $A$ be a smooth and
proper  DG algebra
(over a fixed field $k$). A perfect DG bimodule $M\in \Perf
(A^{\op}\otimes A)$ defines the endofunctor $$\Phi
_M=-\stackrel{\bL}{\otimes }_AM:\Perf A\to \Perf A$$
where $\Perf A$ is the triangulated category of perfect DG $A$-modules.
It is natural
to define the "number of fixed points of $\Phi _M$" as the
alternating sum
$\sum _i (-1)^i\dim HH_i(M),$
where $HH_i(M)$ is the i-th Hochschild homology space of the
DG bimodule $M.$

The functor $\Phi _M$ defines the endomorphism $HH_j (\Phi _M)$ of
the Hochschild homology $HH_j (A)$ for each $j\in
\bbZ.$ We prove the following LFP theorem (= Theorem \ref{Lef}) for
$\Phi _M$ (our assumption on $A$ guarantees that all spaces involved
have finite dimension).

\begin{theo} \label{LFP-algebraic}
Let $A$ be a smooth and proper DG algebra over a field
$k$ and let $M\in \Perf A^{\op}\otimes A$ be a perfect DG bimodule.
Then there is an equality of the two elements of $k$
$$\sum _i (-1)^i\dim HH_i(M)=\sum _j (-1)^j\Tr HH_j (\Phi
_M).$$
\end{theo}

Actually a proof of this theorem (but not the statement) is
essentially contained in a beautiful preprint \cite{Shk} of D.
Shklyarov, where the Hirzebruch-Riemann-Roch (HRR) theorem for DG
algebras is discussed. It turns out that the proofs of theorems HRR and LFP
have much in common. Since
the paper \cite{Shk} seems to be unfortunately unpublished we
thought it worthwhile to give a simultaneous presentation of
theorems HRR and LFP. Thus most of what is contained in Part II is
from \cite{Shk}.

This paper has two parts.
It was our initial plan to deduce Theorem \ref{LFP-geometric} from
Theorem \ref{LFP-algebraic} using the description of the derived category
$D^b(X)$ as the category $\Perf A$ of perfect modules over a smooth and
proper DG algebra $A.$ But then we found a short self-contained proof of
Theorem \ref{LFP-geometric}, so the two parts of this paper are
completely independent (but parallel).

It is our pleasure to thank Mike Mandell for teaching us some algebraic
topology exercises. Laurentiu Maxim suggested to us Theorem
\ref{lefschetz-theorem-for-two-maps} as a generalization of Theorem
\ref{LFP-singular-homology}. Damien Calaque and Christopher Deninger
 asked the right questions and provided useful comments on the first
 version of this paper. We also thank Mircea Mustata, William Fulton
 and Ajay Ramadoss for useful discussions of the subject.

\part{Lefschetz fixed point theorem for Fourier-Mukai functors}

\section{Fourier-Mukai functors}
We fix a field $k.$ All our varieties will be $k$-varieties.

If $Z$ is a smooth projective variety we denote by
$D^b(Z)=D^b(cohZ)$ the bounded derived category of coherent sheaves
on $Z.$

Let $X$ and $Y$ be smooth and projective varieties over $k.$ An object
$E\in D^b(X\times Y)$ defines the corresponding {\it Fourier-Mukai}
functor $\Phi _E:D^b(X)\to D^b(Y)$ by the formula
$$\Phi _E(-)=\bR p_*(E\stackrel{\bL}{\otimes }q^*(-)),$$
where $p$ and $q$ are the projections
$$X\stackrel{q}{\leftarrow} X\times Y\stackrel{p}{\to} Y$$

Denote by $\Delta :X\to X\times X$ the diagonal morphism.
The object $\Delta _*\cO _X\in D^b(X\times X)$ induces the identity
functor $\id =\Phi _{\Delta _*\cO _X}:D^b(X)\to D^b(X).$

Given another smooth projective variety $Z$ and
$E^\prime \in D^b(Y\times Z)$ the composition of functors $\Phi _{E^\prime}
\cdot \Phi _{E}$ is isomorphic to the functor $\Phi _{E^\prime *E},$
where $E^\prime *E\in D^b(X\times Z)$ is the usual convolution of $E^\prime$
and $E$ \cite{Mu}.

The functor $\Phi _E$ induces the linear map $H(\Phi _E)$
between $H(X)$ and $H(Y),$ where $H(-)$ denotes the Hochschild
homology or the singular cohomology (if $k=\bbC$).
We are going to prove LFP type theorems for
these operators $H(\Phi _E).$ Later in Section
\ref{section-singular-homology}
we prove an analogous theorem for singular homology.

\section{Hochschild homology of smooth projective varieties}\label{Hochschild-homology-section}

Let $X$ be a smooth projective variety of dimension $n.$
We recall one of the many possible (equivalent)
definitions of the Hochschild homology of $X$ following
\cite{Cal1}. Namely let $S^{-1}_X=\omega ^*_X[-n]\in D^b(X)$
denote the shift of the dual of the canonical line bundle of $X.$
Consider the diagonal embedding $\Delta :X\to X\times X.$
Then one defines
$$HH_i(X)=
\Hom ^{-i}_{D^b(X\times X)}(\Delta _* S^{-1}_X,\Delta _*\cO _X).$$
We put
$$HH(X)=HH_\bullet (X)=\bigoplus _iHH_i(X).$$
Alternatively
$$HH_i(X)=\Hom ^{-i}_{D^b(X)}(\cO _X,\bL \Delta ^*\Delta _*\cO _X).$$

Actually we will never need to use the definition of $HH (X)$
but rather some of its properties which we now summarize.

\medskip

\noindent{\bf Properties of $HH (X).$}

\medskip

\noindent{1.} $\dim HH(X)<\infty $ and $HH(\pt)=
HH_0(\pt)=k.$

\noindent{2.} An object $E\in D^b(X\times Y)$ defines a degree
preserving linear map $HH (\Phi _E):HH (X)
\to HH (Y).$ In particular if $f:X\to Y$ is a morphism, then
the structure sheaf of its graph $\cO _{\Gamma (f)}$ considered as
an object of $D^b(X\times Y)$ or $D^b(Y\times X)$ defines the
corresponding linear maps which we denote
$$f_*:HH (X)\to HH (Y),\quad f^*:HH (Y)
\to HH(X).$$

The linear map $HH (\Phi _{\Delta _*\cO _X})$
defined by the object $\Delta _*\cO _X\in D^b(X\times X)$
is the identity.

\noindent{3.} The correspondence $E \mapsto HH (\Phi _E)$ is
functorial: Given $E^\prime \in D^b(Y\times Z)$ the convolution
$E^\prime * E \in D^b(X\times Z)$ defines the map
$HH (\Phi _{E^\prime *E})$ which is the composition
$HH (\Phi _{E^\prime })\cdot HH (\Phi _E).$

\noindent{4.} There exists the canonical {\it Kunneth isomorphism}
$$K:HH(X)\otimes HH(Y)
\longrightarrow HH(X\times Y)$$

\noindent{5.} If $\sigma :X\times X \to X\times X$ denotes the
transposition then the induced  map
$$K^{-1}\cdot \sigma _* \cdot K:HH (X)\otimes HH (X)
\to HH (X)\otimes HH (X)$$
is $a\otimes b\mapsto (-1)^{\deg (a)\deg (b)}b\otimes a.$

\noindent{6.} Let $E\in D^b(X\times Y)$ and $E^\prime \in D^b(Z\times W).$
Then the following diagram commutes
$$\begin{array}{rclcl}
HH (X) & \otimes & HH(Z) & \stackrel{K}{\longrightarrow}
& HH(X\times Z) \\
HH(\Phi _E )\downarrow & & \downarrow HH (\Phi _{E^\prime})   & &
\downarrow HH (
\Phi _{E\boxtimes E^\prime}) \\
HH(Y) & \otimes & HH (W) & \stackrel{K}{\longrightarrow}
& HH(Y\times W)
\end{array}
$$

\noindent{7. {\it The Euler class.}} Consider an object $N\in D^b(X)$ as
an object in $D^b(\pt \times X).$ Define the  Euler class of $N$
as $Eu (N)=HH (\Phi _N)(1)\in HH_0(X).$
The map $Eu$ descends to a group homomorphism
$$Eu :K_0(D^b(X))\to HH_0(X).$$
Given $E \in
D^b(X\times Y)$
the following diagram commutes
$$\begin{array}{ccc}
D^b(X) & \stackrel{\Phi _E}{\longrightarrow} & D^b(Y) \\
\downarrow Eu & & \downarrow Eu \\
HH(X) & \stackrel{HH(\Phi _E)}{\longrightarrow} & HH(Y)
\end{array}
$$
If $X=\pt$ and hence $N$ is just a complex of vector spaces then
$$Eu(N)=\sum _i(-1)^i\dim H^i(N)\in HH (pt)=k.$$

Property 1 is clear; 2,3,7 are from \cite{Cal1} (the Euler class is called
the Chern character in \cite{Cal1}) and 4,6
are from \cite{MaSte}. The property 5 follows from the usual
supercommutativity of the tensor product of complexes.

\medskip

\noindent{\it Pairing on $HH.$} We want to consider the
following pairing on $HH(X).$

\begin{defi} \label{natural-paring-geometric}
Consider the diagram of morphisms
$$X\times X \stackrel{\Delta}{\leftarrow }X \stackrel {p}{\to} \pt$$
and define the map $\langle ,\rangle _X:HH (X)\otimes
HH (X) \to k$ as the composition
$$HH (X)\otimes HH (X)\stackrel{K}{\to}
HH (X\times X)\stackrel {HH(\Delta ^*)}{\to} HH (X)
\stackrel{HH(p_*)}{\to }HH (\pt)=k.$$
\end{defi}

\begin{remark} It follows from Property 5 above that
$\langle a,b\rangle _X=(-1)^{\deg (a)\deg (b)}\langle b,a\rangle_X.$
\end{remark}

\begin{remark} Apparently this pairing is not the same as the Mukai pairing
considered by Caldararu \cite{Cal1}, although the two are closely
related (see \cite{Ram2}).
Our pairing is a direct analogue of the pairing \ref{def-pairing}
below and the
next lemma (and its proof) is similar to  Lemma \ref{main-lemma}
\end{remark}

\medskip

\noindent{\it Notation.} Given smooth projective varieties $X,Y,Z,W$
and objects $E\in D^b(X\times Y), E^\prime \in D^b(Z\times W)$ we denote
the functor $\Phi _{E \boxtimes E^\prime}:D^b(X\times Z)\to D^b(Y\times W)$
by $\Phi _E\boxtimes \Phi _{E^\prime}.$

\medskip

Let $X$ and $Y$ be smooth projective varieties and let
$E\in D ^b(X\otimes Y).$ Define
$$Eu (E)^\prime \in HH (X)
\otimes HH(Y)$$ to be the inverse image of $Eu(E)$
under the Kunneth isomorphism.

\begin{lemma} \label{main-geometric}
In the above notation the linear map $HH (\Phi _E):
HH(X)\to HH(Y)$ is the convolution with the class
$Eu ^\prime (E).$ I.e. it is equal to the composition
$$HH (X)\stackrel{\id \otimes Eu ^\prime (E)}
{\longrightarrow} HH (X)\otimes HH(X) \otimes
HH(Y)\stackrel{\langle ,\rangle _X \otimes \id}{\longrightarrow}
HH(Y).$$
\end{lemma}

\begin{proof} First notice that the Fourier-Mukai functor $\Phi _E:D^b(X)
\to D^b(Y)$ is isomorphic to the following composition of functors
$$\begin{array}{ccccccc}
D^b(X\times \pt ) & \stackrel{\id \boxtimes \Phi _E}{\longrightarrow} &
D^b(X\times X \times Y) & \stackrel{\Delta ^* \boxtimes \id }
{\longrightarrow } & D^b(X\times Y) &
\stackrel{p_*\boxtimes \id}{\longrightarrow} & D^b(\pt \times Y)\\
\vert \vert & & & & & & \vert \vert \\
D^b(X) & & & & & & D^b(Y)
\end{array}
$$

Now it follows form properties 3,4,6,7 above that the corresponding
linear map $HH(\Phi _E):HH(X)\to HH(Y)$ is
equal to the following composition
$$\begin{array}{ccccc}
HH(X) & & & &  \\
\vert \vert & & & &  \\
HH(X)\otimes HH(\pt ) & \stackrel{\id \otimes Eu(E)}
{\longrightarrow} &
HH(X)\otimes HH(X\times Y) & & HH(Y)
 \\
 & & \downarrow \id \otimes K^{-1} & & \vert \vert \\
 & &
HH(X)\otimes HH(X)\otimes HH (Y) &  &  HH(pt) \otimes HH(Y) \\
& & \downarrow K\otimes \id &  &  \uparrow HH (p_*)\otimes \id \\
& & HH(X\times X)\otimes HH(Y) & \stackrel
{HH (\Delta ^*)\otimes \id }{\longrightarrow} &
HH(X)\otimes HH(Y)
\end{array}
$$

Finally notice that the composition   $\id \otimes (K^{-1} \cdot Eu(E))$
 is equal to $\id \otimes
Eu ^\prime (E).$ Also the composition
$(HH(p_*)\cdot HH(\Delta ^*) \cdot K)\otimes \id$
is the map
$\langle ,\rangle _X \otimes \id$. This proves the lemma.
\end{proof}

\begin{cor} \label{nondeg-pairing}
The pairing $\langle ,\rangle _X$ is nondegenerate.
\end{cor}

\begin{proof} The Fourier-Mukai functor $\Phi _{\Delta _*\cO _X}:
D^b(X)\to D^b(X)$ is isomorphic to the identity. Hence it follows from
Lemma \ref{main-geometric} that the map $HH(X)\to HH(X)^*$
defined by $\langle ,\rangle _X$ is injective. Since the space
$HH(X)$ is finite dimensional this map is bijective.
\end{proof}

\begin{remark} It follows from Lemma \ref{main-geometric} and
Corollary \ref{nondeg-pairing} that the Euler class
$$Eu^\prime (\Delta _* \cO _X)\in HH(X)\otimes HH(X)$$
"is the pairing" $\langle ,\rangle _X.$ That is, if $\{ e_i\}$
is a basis of $HH(X)$ and $\{ f_i\}$ is the right-dual basis (i.e.
$\langle e_i,f_j\rangle _X=\delta _{ij}$) then $Eu^\prime (\Delta _* \cO _X)=
\sum f_i\otimes e_i.$
\end{remark}

\begin{defi}\label{general-hochschild-homology}
 Let $E\in D^b(X\times X)$ and consider again the diagram
$$X\times X \stackrel{\Delta}{\leftarrow }X \stackrel {p}{\to} \pt.$$
Define the $i$-th Hochschild homology $HH_i(E)$ to be
the space $\Hom ^{-i}_{D^b(X)}(\cO _X, \bL \Delta ^*E).$  Then the total space
$HH_\bullet (E)$ is finite dimensional because $X$ is smooth and proper.
Note that $HH(\Delta _*\cO _X)=HH(X).$
\end{defi}

\begin{lemma} \label{computation} Let $E\in D^b(X\times X)$ and let
$Eu^\prime (E)=\sum _sa_s\otimes b_s\in HH(X)\otimes HH(X).$ Then
$$\sum _i(-1)^i\dim HH_i(E)=\sum _s\langle a_s,b_s\rangle _X$$
\end{lemma}

\begin{proof}  Properties 2,3,7 above imply that the following diagram
commutes
$$\begin{array}{ccccccc}
D^b(\pt ) & \stackrel{\Phi _E}{\longrightarrow} & D^b(X\times X) &
\stackrel{\Delta ^*}{\longrightarrow} & D^b(X) & \stackrel{p_*}
{\longrightarrow} & D^b(\pt)\\
\downarrow Eu & & \downarrow Eu & & \downarrow Eu & & \downarrow Eu \\
HH(\pt) & \stackrel{HH(\Phi _E)}{\longrightarrow} & HH(X\times X) &
\stackrel{HH(\Delta ^*)}{\longrightarrow} & HH(X) & \stackrel{HH(p_*)}
{\longrightarrow} & HH(\pt)
\end{array}
$$
Now the lemma immediately follows from Definitions
\ref{natural-paring-geometric},
\ref{general-hochschild-homology} and the last part in Property 7.
\end{proof}

We are ready for the geometric Hochschild homology version of the
Lefschetz fixed point theorem.

\begin{theo} \label{LFP-Hochschild-homology-geometric-in-text}
Let $X$ be a smooth projective variety over a field $k$
and $E\in D^b(X\times X).$ For each $j$ consider the linear
endomorphism $HH_j(\Phi _E)$ of $HH_j(X).$ Then there is the equality
$$\sum _i(-1)^i\dim HH_i(E)=\sum _j(-1)^j\Tr HH_j(\Phi _E).$$
\end{theo}

\begin{proof} Choose a homogeneous basis $\{v_m\}$ of $HH_\bullet (X)$
and let
$\{\bar{v}_m\}\subset HH_\bullet (X)$ be the left-dual basis with
respect to $\langle ,\rangle _X$, i.e.
$\langle \bar{v}_m,v_n\rangle _{X}=\delta _{ mn}.$
Let
$$Eu (E)^\prime =\sum _{m,n}\alpha _{mn}\cdot \bar{v}_m\otimes v_n$$
for $\alpha _{mn}\in k.$

Then by Lemma \ref{computation}
$$\sum _i(-1)^i \dim HH_i(E)=\sum _{m}\alpha _{mm}.$$
On the other hand by Lemma \ref{main-geometric}
$$HH (\Phi _E)(v_l)=
\sum _{m,n}\alpha _{mn}\cdot \langle v_l,\bar{v}_m\rangle _X\cdot v_n$$
By Property 5 above
$$\langle v_l,\bar{v}_m\rangle _X=(-1)^{\deg (v_l)\deg (\bar{v}_m)}
\langle \bar{v}_m,v_l\rangle _{X}=(-1)^{\deg (v_l)}
\langle \bar{v}_m,v_l\rangle _{X}=(-1)^{\deg (v_l)}\delta _{lm}.$$
So the trace of the linear operator
$HH(\Phi _E)$ on $HH (X)$ equals
$\sum _m(-1)^{\deg (v_m)}\alpha _{mm}.$
And its supertrace is
$$\sum _j(-1)^j\Tr HH_j(\Phi _E)=\sum _m\alpha _{mm}$$
which proves the theorem.
\end{proof}

Next we want to discuss the Hirzebruch-Riemann-Roch (HRR) theorem for
Hochschild homology, which is closely related to the Lefschetz fixed point
theorem.

\begin{defi} \label{def-of-dot}
Let $E,F\in D^b(Y).$ We define the integer
$$E\cdot F :=\sum _j(-1)^j\dim H ^j(Y,E\stackrel{\bL}{\otimes }F).$$
It may be called the {\rm intersection} of $E$ and $F.$
\end{defi}

\begin{remark} \label{coincidence-of-dots}
If $Y=X\times X$ and $F$ is the structure sheaf of the diagonal
 we have $E\cdot F=\sum _i(-1)^iHH_i(E).$
 In particular $\Delta _*\cO _X \cdot \Delta _*\cO _X=
\sum _i(-1)^iHH_i(X).$
\end{remark}

The next proposition is the HRR theorem for Hochschild homology.

\begin{prop} \label{HRR-HH-geometric}
Let $Y$ be a smooth projective variety and $E,F\in D^b(Y).$ Then
$$E\cdot F=\langle Eu(E),Eu(F)\rangle _Y.$$
\end{prop}

\begin{proof}
The diagram
$$\begin{array}{ccccccc}
D^b(\pt ) & \stackrel{\Phi _{E\boxtimes F}}{\longrightarrow} & D^b(Y\times Y) &
\stackrel{\Delta ^*}{\longrightarrow} & D^b(Y) & \stackrel{p_*}
{\longrightarrow} & D^b(\pt)\\
\downarrow Eu & & \downarrow Eu & & \downarrow Eu & & \downarrow Eu \\
HH(\pt) & \stackrel{HH(\Phi _{E\boxtimes F})}{\longrightarrow} & HH(Y\times Y) &
\stackrel{HH(\Delta ^*)}{\longrightarrow} & HH(Y) & \stackrel{HH(p_*)}
{\longrightarrow} & HH(\pt)
\end{array}
$$
commutes by Property 7 above. By definition the number $E\cdot F$ is equal to
the Euler characteristic of the complex of vector spaces $p_*\cdot \Delta ^* \cdot
\Phi _{E\boxtimes F}(k).$ Hence, by the last part of Property 7, it is equal to
$$E\cdot F=Eu \cdot p_*\cdot \Delta ^* \cdot
\Phi _{E\boxtimes F}(k).$$

On the other hand, by definition of the Euler class and Property 6 we have
$$HH(\Phi _{E\boxtimes F})\cdot Eu(k)=K(Eu(E)\otimes Eu(F)).$$
It follows that
$$\langle Eu(E),Eu(F)\rangle _Y=HH(p_*)\cdot HH (\Delta ^*)\cdot HH(\Phi _{E\boxtimes F})\cdot Eu (k).$$
This proves the proposition.
\end{proof}

\section{Singular cohomology of smooth complex projective varieties}

Assume now that  $k=\bbC.$

Let $X$ be a smooth complex projective variety and consider the singular
cohomology $H^\bullet (X,\bbC).$ It has the Hodge decomposition
$$H^i(X,\bbC )=\bigoplus _{p+q=i}H^p(X,\Omega ^q).$$
Let $Y$ be another smooth complex projective variety and $f:X\to Y$ be
a morphism. There is the induces degree preserving morphism on cohomology
$$f^*:H^\bullet (Y,\bbC)\to H^\bullet (X,\bbC),$$
and hence by Poincare duality
$H^i(X,\bbC )^*\simeq H^{2\dim (X)-i}(X,\bbC)$ and
$H^i(Y,\bbC )^*\simeq H^{2\dim (X)-i}(Y,\bbC)$ the map
$$f_*:H^\bullet (X,\bbC)\to H^{\bullet +2\dim (Y)-2\dim (X)}(Y,\bbC).$$

Consider the projections $X\stackrel{q}{\leftarrow}X\times Y\stackrel{p}
{\to }Y.$ Then any class $\alpha \in H^\bullet (X\times Y)$ defines
the corresponding convolution map
$$H^\bullet (X,\bbC)\to H^\bullet (Y,\bbC),\quad \beta \mapsto
p_*(\alpha \cup q^*\beta).$$

For any object $S\in D^b(X)$ there is its Chern character
$$ch(S)\in\bigoplus _pH^p(X,\Omega ^p _X)\subset H^\bullet (X,\bbC).$$
Recall also the Todd class $td_X\in \oplus _pH^p(X,\Omega ^p _X)$
 and its square root $\sqrt{td_X}$ (which is uniquely defined if
 one requires its degree zero term to be 1).

\begin{defi} For any $S\in D^b(X)$ its
Mukai vector $$\upsilon (S)\in \bigoplus _pH^p(X,\Omega ^p _X)$$
is the element $\upsilon (S)=ch(S)\cup \sqrt{td_X}.$ This gives
a map $\upsilon :D^b(X)\to H^\bullet (X,\bbC)$ from objects of the derived
category to the singular cohomology.
\end{defi}

\begin{defi} Any object $E\in D^b(X\times Y)$ defines the linear map
$H^\bullet(\Phi _E):H^\bullet (X,\bbC)\to H^\bullet (Y,\bbC)$ which is the
convolution with the Mukai vector $\upsilon (E)$
$$H^\bullet (\Phi _E)(\beta)=p_*(\upsilon (E)\cup q^*\beta).$$
\end{defi}

It follows from the Grothendieck-Riemann-Roch theorem that the following
diagram commutes
$$\begin{array}{ccc}
D^b(X) & \stackrel{\Phi _E}{\longrightarrow} & D^b(Y)\\
\upsilon \downarrow & & \downarrow \upsilon \\
H^\bullet (X,\bbC) & \stackrel{H^\bullet (\Phi _E)}{\longrightarrow }
& H^\bullet (Y,\bbC)
\end{array}
$$
The map $H^\bullet (\phi _E)$ does not preserve the degree of the cohomology
but it preserves the Hodge verticles
$$\bigoplus _{p-q=\text{fixed}}H^p(X,\Omega ^q).$$
Hence $H^\bullet (\Phi _E)$ preserves the parity of the degree
of the cohomology,
i.e. it is the direct sum of operators $H^{\text{ev}}(\Phi _E)$ and
$H^{\text{odd}}(\Phi _E).$

Next is the singular cohomology version of LFP theorem for
Fourier Mukai transforms.

\begin{theo} \label{LFP-singular-cohomology-in-text}
Let $X$ be a smooth complex projective variety and
let $E\in D^b(X\times X).$ Consider the induced linear operators
$$H^{\text{ev}}(\Phi _E):H^{\text{ev}}(X,\bbC)\to H^{\text{ev}}(X,\bbC)\quad
\text{and}\quad H^{\text{odd}}(\Phi _E):H^{\text{odd}}(X,\bbC)\to
H^{\text{odd}}(X,\bbC).$$ Then there is the equality
$$\sum _i(-1)^i\dim HH_i(E)=\Tr H^{\text{ev}}(\Phi _E)-
\Tr H^{\text{odd}}(\Phi _E).$$
\end{theo}

\begin{proof} We deduce this theorem from Theorem
\ref{LFP-Hochschild-homology-geometric-in-text}.

Since $X$ is a smooth variety over a field of characteristic zero one has
the Hochschild-Kostant-Rosenberg isomorphism
$$I^X_{HKR}:HH_\bullet (X)\stackrel{\sim}{\longrightarrow }
\bigoplus _{p,q} H^p(X,\Omega ^q_X),$$
which identifies the space $HH_i(X)$ with the Hodge vertical
$\bigoplus _{p-q=-i}H^p(X,\Omega _X^q).$ Denote by $I^X$ the composition
$$I^X:HH_\bullet (X)\stackrel{I^X_{HKR}}{\longrightarrow}
\bigoplus _{p,q}H^p(X,\Omega ^q_X)\simeq H^\bullet (X,\bbC)\stackrel
{\cup \sqrt{td_X}}{\longrightarrow} H^\bullet(X,\bbC)$$

We will use the following important result from \cite{MaSte}, Thm.1.2 (which
in turn is heavily based on the work of Caldararu \cite{Cal1},\cite{Cal2}  and Ramadoss
\cite{Ram}).

\begin{theo} \label{Hochschild=singular}
Let $X$ and $Y$ be smooth complex projective varieties and
$\cE\in D^b(X\times Y).$ Then the following diagram commutes
$$\begin{array}{ccc}
HH (X) & \stackrel{HH  (\Phi _{\cE})}{\longrightarrow}
& HH(Y) \\
I^X \downarrow & & \downarrow I^Y\\
H^\bullet (X,\bbC ) & \stackrel{H^\bullet (\Phi _{\cE})}{\longrightarrow} &
H^\bullet (Y,\bbC)
\end{array}
$$
\end{theo}
We apply this theorem in case $X=Y$ and $\cE =E.$ Notice that $I^X$ is an isomorphism,
which preserves the parity of the cohomology space, i.e. it is an isomorphism
of $\bbZ /(2)$-graded spaces. This implies that
the supertrace $\Tr H^{\text{ev}}(\Phi _E)-
\Tr H^{\text{odd}}(\Phi _E)$ equals the supertrace
$\sum _j(-1)^j\Tr HH_j(\Phi _E).$ So Theorem \ref{LFP-singular-cohomology-in-text}
 follows from Theorem
\ref{LFP-Hochschild-homology-geometric-in-text}.
\end{proof}

\begin{remark} \label{formula-for-dot-by-class}
Let $Y$ be a smooth complex projective variety and
let $E,F\in D^b(Y).$ The Hirzebruch-Riemann-Roch theorem implies
that
$$E\cdot F =\int _YchE\cup chF\cup td_Y$$
(Definition \ref{def-of-dot}). Let now $Y=X\times X$ and $F=\Delta _*
\cO _X.$ By Grothendieck-Riemann-Roch theorem we have
$$ch(\Delta _*\cO _X)=\Delta _*(td _X)\cup (td_{X\times X})^{-1}$$
Hence by Remark \ref{coincidence-of-dots} and the above formula we get
$$\sum (-1)^i\dim HH_i(E)=\Delta _*\cO _X \cdot E
=\int _{X\times X}chE\cup \Delta _*(td _X)=
\int _X\Delta ^*(chE) \cup td _X.$$
This gives a formula for the left hand side in the LFP theorem
in terms of the Chern character of $E.$
\end{remark}

\section{Singular homology of smooth complex projective varieties}
\label{section-singular-homology}

Let $X$ be a smooth complex projective variety and consider its singular
homology
$$H_\bullet (X)=\bigoplus _jH_j(X,\bbC).$$

Let $f:X\to X$ be a morphism. Then one has the induced
linear maps $H_j(f):H_i(X,\bbC)\to
H_j(X,\bbC).$ Denote by $\Gamma (f)\subset X\times X$ the graph of the
morphism $f$ and consider its structure sheaf $\cO _{\Gamma (f)}$
is an object in $D^b(X\times X).$ Next is the version of LFP
theorem for singular homology.

\begin{theo} \label{LFP-singular-homology-in-text}
In the previous notation there is the equality
$$\sum _i(-1)^i\dim HH_i(\cO _{\Gamma (f)})=\sum _j(-1)^j\Tr H_j(f).$$
\end{theo}

\begin{remark}
In case the graph $\Gamma (f)$ intersects the diagonal
$\Delta \subset X\times X$ transversally (hence at a finite number of
points) we have
$\sum _i(-1)^iHH_i(\cO _{\Gamma (f)})=HH_0(\cO _{\Gamma (f)})$
and $\dim HH_0(\cO _{\Gamma (f)})$ is the number of fixed points of $f.$
So one recovers the classical LPF theorem.
\end{remark}

\begin{proof} We deduce Theorem \ref{LFP-singular-homology-in-text} from
Theorem \ref{LFP-singular-cohomology-in-text} using Poincare duality.

Namely by Theorem \ref{LFP-singular-cohomology-in-text} the number
$\sum _i(-1)^iHH_i(\cO _{\Gamma (f)})$ equals the supertrace of the linear
operator $H^\bullet (\Phi _{\cO _{\Gamma (f)}}):H^\bullet (X,\bbC )
\to H^\bullet (X,\bbC).$ On the other hand using the Poincare duality
$H^{2n-j}(X,\bbC) \simeq H_j(X,\bbC)$ the map $H_\bullet (f)$ induces
the map $f_*:H^\bullet (X,\bbC)\to H^\bullet (X,\bbC).$ Notice that $f_*$
preserves the degree of the cohomology, i.e. it is the sum of maps
$f_*^s:H^s(X,\bbC)\to H^s(X,\bbC).$ Thus it suffices to prove that
\begin{equation} \label{*}
{\sum _i(-1)^iHH_i(\cO _{\Gamma (f)})=\sum _s(-1)^s\Tr f_*^s}
\end{equation}

So it remains to compare the linear maps
$H^\bullet (\Phi _{\cO _{\Gamma (f)}})$ and $f_*$ and show that
their supertraces are equal. This is achieved in the next lemma.

\begin{lemma} \label{maps-on-cohomology}
Let $X$ and $Y$ be complex projective varieties and
$g:X\to Y$ be a morphism. Then the following diagram commutes
$$\begin{array}{ccc}
H^\bullet (X,\bbC ) & \stackrel{H^\bullet (\Phi _{\cO _{\Gamma (g)}})}
{\longrightarrow} & H^\bullet (Y,\bbC) \\
\downarrow \cup \sqrt{td_X}  & & \downarrow \cup \sqrt{td_Y} \\
H^\bullet (X,\bbC) & \stackrel{g_*}{\longrightarrow} & H^\bullet (Y,\bbC)
\end{array}
$$
\end{lemma}

The theorem follows from the lemma (applied to the case $Y=X$ and $g=f$)
because the operator $\cup \sqrt{td_X}$
is an isomorphism which
preserves the parity
of the cohomology, so the supertraces of
$H^\bullet (\Phi _{\cO _{\Gamma (f)}})$ and $f_*$ are equal.
Thus it remains
to prove the lemma.

\begin{proof} Consider the diagram
$$
\xymatrix@1{X \ar@<.4ex>[r]^{i} & X\times Y \ar@<.4ex>[l]^{q}
\ar@<.4ex>[r]^{p} & Y}
$$
where $p$ and $q$ are the two projections and $i:X\to X\times Y$
is the isomorphism
of $X$ onto the graph $\Gamma (g)$ (so that $g=p\cdot i$). By definition
$$H^\bullet (\Phi _{\cO _{\Gamma (g)}})(-)=
p_*(ch(i_*\cO _X)\cup \sqrt{td _{X\times Y}})\cup q^*(-)).$$
By Grothendieck-Riemann-Roch theorem
$$ch(i_*\cO _X)\cup td_{X\times Y}=i_*(ch(\cO _X)\cup td_X)=i_*(td_X)$$
hence
$$ch(i_*\cO _X)=i_*(td_X)\cup (td _{X\times Y})^{-1}$$
and
$$\begin{array}{rcl}
H^\bullet (\Phi _{\cO _{\Gamma (g)}})(-) & = &
p_*(i_*(td_X)\cup (\sqrt{td_{X\times Y}})^{-1}\cup q^*(-)) \\
 & = & p_*(i_*(td_X)\cup q^*(\sqrt{td_{X}})^{-1}\cup q^*(-)\cup p^*(\sqrt{td _Y})^{-1}) \\
 & = & p_*(i_*(td_X)\cup q^*((\sqrt{td_{X}})^{-1}\cup -))\cup (\sqrt{td _Y})^{-1}
\end{array}
$$
Since $i^*q^*=\id$ we have for any $\beta \in H^\bullet (X,\bbC)$
$$i_*(td_X)\cup q^*\beta =i_*(td_X\cup i^*q^*\beta )=i_*(td_X\cup \beta)$$
Therefore
$$\begin{array}{rcl}
H^\bullet (\Phi _{\cO _{\Gamma (g)}})(-) & = &
p_*(i_*(td_X\cup (\sqrt{td_X})^{-1}\cup -))\cup (\sqrt{td_Y})^{-1}\\
 & = & p_*i_*(\sqrt{td_X}\cup -)\cup (\sqrt{td _Y})^{-1}\\
 & = & f_*(\sqrt{td_X}\cup -)\cup (\sqrt{td _Y})^{-1}
\end{array}
$$

This proves the lemma and Theorem \ref{LFP-singular-homology-in-text}
\end{proof}
\end{proof}

\section{ Lefschetz fixed point theorem for two maps}\label{two-maps}

In this section we prove a generalization of Theorem
\ref{LFP-singular-homology-in-text} for two maps between different
varieties of the same dimension. Namely, let $X$ and $Y$ be two smooth
complex projective varieties and $f,g:X\to Y$ be morphisms. We obtain the induced
maps
$$f_*:H_i(X)\to H_i(Y),\quad g^*:H^j(Y)\to H^j(X).$$
Assume now that $\dim X =\dim Y=d.$ Then we get the diagram of maps
$$\begin{array}{ccc}
H_i(X) & \stackrel{f_*}{\longrightarrow} & H_i(Y)\\
D\uparrow & & \downarrow D \\
H^{2d-i}(X) & \stackrel{g^*}{\longleftarrow} & H^{2d-i}(Y)
\end{array}
$$
where $D$ denotes the Poincare duality isomorphisms on $X$ and $Y.$
We want a formula for the supertrace of the composition
$$D\cdot g^* \cdot D\cdot f_*:H_\bullet (X)\to H_\bullet (X).$$

Note that the composition $D\cdot f_*\cdot D:H^j(X)\to H^j(Y)$ is nothing
but the push forward map $f_*$ on cohomology which we considered in Section
\ref{section-singular-homology} above. Hence the supertrace of the
composition $D\cdot g^* \cdot D\cdot f_*$ equals the supertrace of the
composition
$$H^\bullet(X)\stackrel{f_*}{\longrightarrow} H^\bullet(Y)
\stackrel{g^*}{\longrightarrow}H^\bullet(X)$$
We denote by $H^j(g^*\cdot f_*):H^j(X)\to H^j(X)$ the restriction
of this last composition to j-th cohomology.

\begin{theo} \label{lefschetz-theorem-for-two-maps}
Let $X$ and $Y$ be smooth complex projective varieties of
the same dimension and let $f,g:X\to Y$ be two regular maps.
 Then in the previous notation there is the equality
$$\cO _{\Gamma (f)}\cdot \cO _{\Gamma (g)}=
\sum _j(-1)^j\Tr H^j(g^*\cdot f_*).$$
Hence the number $\cO _{\Gamma (f)}\cdot \cO _{\Gamma (g)}$ also equals
the supertrace of the map $D\cdot g^* \cdot D\cdot f_*$ on homology of
$X.$
\end{theo}

Clearly Theorem \ref{LFP-singular-homology-in-text} is a special case of
Theorem \ref{lefschetz-theorem-for-two-maps} with $X=Y$ and $g=\id.$

\begin{proof} We give the proof in two steps. In the first one
we consider Hochschild homology and in the second - singular cohomology.

\medskip

\noindent{Step 1.}
Consider the graph $\Gamma (g)\subset X\times Y$ as a
subvariety of $Y\times X.$ We have the functors
\begin{equation}\label{functors}
\Phi _{\cO _{\Gamma (f)}}:D^b(X)\to D^b(Y),\quad \Phi _{\cO _{\Gamma (g)}}:
D^b(Y)\to D^b(X)\end{equation}
which induce linear maps
$$HH(\Phi _{\cO _{\Gamma (f)}}):HH(X)\to HH(Y),\quad
HH(\Phi _{\cO _{\Gamma (g)}}):HH(Y)\to HH(X)$$

By Property 3 in Section \ref{Hochschild-homology-section}
their composition equals
\begin{equation}\label{composition} HH(\Phi _{\cO _{\Gamma (g)}})
\cdot HH(\Phi _{\cO _{\Gamma (f)}})=HH(\Phi _{\cO _{\Gamma (f)}*\cO _{\Gamma (g)}})
\end{equation}
where $\cO _{\Gamma (f)}*\cO _{\Gamma (g)}\in D^b(X\times X)$ is
the convolution
of $\cO _{\Gamma (f)}$ and $\cO _{\Gamma (g)}.$ By Theorem
\ref{LFP-Hochschild-homology-geometric-in-text} and Remark \ref{coincidence-of-dots}
we have
\begin{equation}\label{HH1} (\cO _{\Gamma (f)}*\cO _{\Gamma (g)})\cdot \Delta _*\cO _X=
\sum _j(-1)^j\Tr HH_j(\Phi _{\cO _{\Gamma (f)}*\cO _{\Gamma (g)}}).
\end{equation}

\begin{lemma} Let $E,F\in D^b(X\times Y).$ We also consider $F$
as an object in $D^b(Y\times X).$ Then
$$(E*F)\cdot \Delta _*\cO _X=E\cdot F.$$
In particular there is the equality
$(\cO _{\Gamma (f)}*\cO _{\Gamma (g)})\cdot \Delta _*\cO _X=
\cO _{\Gamma (f)}\cdot \cO _{\Gamma (g)}.$
\end{lemma}

\begin{proof}
Consider the obvious diagram
\begin{equation}\label{big-diagram} \begin{array}{ccccc}
X\times Y\times Y\times X & & & &\\
\uparrow \Delta _Y & & & & \\
X\times Y\times X & \stackrel{p^Y}{\longrightarrow} & X\times X & & \\
\uparrow \Delta _X & & \uparrow \Delta _X & & \\
X\times Y & \stackrel{p^Y}{\longrightarrow} & X &
\stackrel{p^X}{\longrightarrow} & \pt
\end{array}
\end{equation}
By definition the convolution $E*F\in D^b(X\times X)$ is the object
$\bR p^Y_*\cdot \bL \Delta _Y^*(E\boxtimes F),$ and the number
$(E*F)\cdot \Delta _*\cO _X$ is the Euler characteristic of the complex
$$\bR p^X_*\cdot \bL \Delta _X^*(E*F)=
\bR p^X_*\cdot \bL \Delta _X^*\cdot \bR p^Y_*\cdot \bL \Delta _Y^*
(E\boxtimes F).$$

Since the square part of diagram \ref{big-diagram} is cartesian, the map
$p^Y$ is smooth and all the varieties are smooth projective it follows
from Lemma 1.3 in \cite{BO} that there is a base change isomorphism of
functors
$$\bL \Delta _X^*\cdot \bR p_*^Y\simeq \bR p_*^Y\cdot \bL \Delta ^*_X.$$
Hence
$$\bR p^X_*\cdot \bL \Delta _X^*(E*F)=
\bR p^X_*\cdot \bR p^Y_*\cdot \bL \Delta _X^*\cdot \bL \Delta _Y^*
(E\boxtimes F).$$

Denote $\Delta :=\Delta _Y\cdot \Delta _X:X\times Y\to
X\times Y\times X\times Y$ ($\Delta$ is the diagonal embedding) and
$p:=p^X\cdot p^Y:X\times Y\to \pt.$ We have
$$\bR p^X_*\cdot \bL \Delta _X^*(E*F)=
\bR p_*\cdot \bL \Delta ^*(E\boxtimes F)=
\bR p_*(E\stackrel{\bL}{\otimes }F)$$
and $E\cdot F$ is the Euler characteristic of the complex
$\bR p_*(E\stackrel{\bL}{\otimes }F).$ This proves the lemma.
\end{proof}

The lemma and the equality \ref{HH1} imply that
\begin{equation}\label{HH2} \cO _{\Gamma (f)}\cdot \cO _{\Gamma (g)}=
 \sum _j(-1)^j\Tr HH_j(\Phi _{\cO _{\Gamma (f)}*\cO _{\Gamma (g)}}).
\end{equation}

\medskip

\noindent{Step 2.} The functors
$\Phi _{\cO _{\Gamma (f)}}, \Phi _{\cO _{\Gamma (g)}}$
also induce the linear maps
$$H^\bullet(\Phi _{\cO _{\Gamma (f)}}):
H^\bullet(X,\bbC)\to H^\bullet(Y,\bbC)\quad
H^\bullet(\Phi _{\cO _{\Gamma (g)}}):
H^\bullet(Y,\bbC)\to H^\bullet(X,\bbC)$$
and by Theorem \ref{Hochschild=singular} the diagram
$$\begin{array}{ccccc}
HH (X) & \stackrel{HH  (\Phi _{\cO _{\Gamma (f)}})}{\longrightarrow}
& HH(Y) & \stackrel {HH  (\Phi _{\cO _{\Gamma (g)}})}{\longrightarrow} &
HH(X)\\
\downarrow I^X & & \downarrow I^Y & & \downarrow I^X \\
H^\bullet (X,\bbC) &
\stackrel{H^\bullet (\Phi _{\cO _{\Gamma (f)}})}{\longrightarrow} &
H^\bullet (Y,\bbC) &
\stackrel{H^\bullet (\Phi _{\cO _{\Gamma (g)}})}{\longrightarrow} &
H^\bullet (X,\bbC)
\end{array}
$$
commutes.
Since the map $I^X$ is an isomorphism which preserves the
parity of the grading it follows that
the supertrace of the composition
\begin{equation}\label{composition}
H^\bullet (\Phi _{\cO _{\Gamma (g)}})\cdot
H^\bullet (\Phi _{\cO _{\Gamma (f)}}):
H^\bullet (X,\bbC)\to H^\bullet (X,\bbC)\end{equation}
 equals
$\cO _{\Gamma (f)}\cdot \cO _{\Gamma (g)}.$
So it suffices to prove the following proposition.

\begin{prop} \label{hochschild-to-singular-for-two-maps}
The supertrace of the composition \ref{composition}
equals the
supertrace
$$\sum _j(-1)^j\Tr H^j(g^*\cdot f_*).$$
\end{prop}

\begin{proof} By Lemma \ref{maps-on-cohomology} above
the following diagram commutes
$$\begin{array}{ccc}
H^\bullet (X,\bbC ) & \stackrel{H^\bullet (\Phi _{\cO _{\Gamma (f)}})}
{\longrightarrow} & H^\bullet (Y,\bbC) \\
\downarrow \cup \sqrt{td_X}  & & \downarrow \cup \sqrt{td_Y} \\
H^\bullet (X,\bbC) & \stackrel{f_*}{\longrightarrow} & H^\bullet (Y,\bbC)
\end{array}
$$
The following lemma is similar.

\begin{lemma} \label{maps-on-cohomology2}
Let $X$ and $Y$ be complex projective varieties and
$g:X\to Y$ be a morphism. Then the following diagram commutes
$$\begin{array}{ccc}
H^\bullet (Y,\bbC ) & \stackrel{H^\bullet (\Phi _{\cO _{\Gamma (g)}})}
{\longrightarrow} & H^\bullet (X,\bbC) \\
\uparrow \cup \sqrt{td_Y}  & & \uparrow \cup \sqrt{td_X} \\
H^\bullet (Y,\bbC) & \stackrel{g^*}{\longrightarrow} & H^\bullet (X,\bbC)
\end{array}
$$
\end{lemma}

\begin{proof} As in the proof of Lemma \ref{maps-on-cohomology} we
consider the diagram
$$
\xymatrix@1{X \ar@<.4ex>[r]^{i} & X\times Y \ar@<.4ex>[l]^{q}
\ar@<.4ex>[r]^{p} & Y}
$$
where $p$ and $q$ are the two projections and $i:X\to X\times Y$
is the isomorphism
of $X$ onto the graph $\Gamma (g)$ (so that $g=p\cdot i$).

By definition
$$H^\bullet (\Phi _{\cO _{\Gamma (g)}})(-)=
q_*(ch(i_*\cO _X)\cup \sqrt{td _{X\times Y}})\cup p^*(-)).$$
By Grothendieck-Riemann-Roch theorem
$$ch(i_*\cO _X)\cup td_{X\times Y}=i_*(ch(\cO _X)\cup td_X)=i_*(td_X)$$
hence
$$ch(i_*\cO _X)=i_*(td_X)\cup (td _{X\times Y})^{-1}$$
and
$$\begin{array}{rcl}
H^\bullet (\Phi _{\cO _{\Gamma (g)}})(-) & = &
q_*(i_*(td_X)\cup (\sqrt{td_{X\times Y}})^{-1}\cup p^*(-)) \\
 & = & q_*(i_*(td_X)\cup q^*(\sqrt{td_{X}})^{-1}\cup p^*(-)
 \cup p^*(\sqrt{td _Y})^{-1}) \\
 & = & q_*(i_*(td_X)\cup q^*(\sqrt{td_{X}})^{-1}\cup p^*(-
 \cup (\sqrt{td _Y})^{-1}))
\end{array}
$$
Since $i^*q^*=\id$ we have for any $\beta \in H^\bullet (X,\bbC)$
$$i_*(td_X)\cup q^*\beta =i_*(td_X\cup i^*q^*\beta )=i_*(td_X\cup \beta)$$
Therefore
$$\begin{array}{rcl}
H^\bullet (\Phi _{\cO _{\Gamma (g)}})(-) & = &
q_*(i_*(td_X\cup (\sqrt{td_X})^{-1}) \cup p^*(-
 \cup (\sqrt{td _Y})^{-1}))\\
 & = & q_*i_*(\sqrt{td_X}\cup i^*p^*(-\cup (\sqrt{td _Y})^{-1}))\\
 & = & \sqrt{td_X}\cup g^*(-\cup (\sqrt{td _Y})^{-1})
\end{array}
$$
This proves the lemma.
\end{proof}

It follows from Lemmas \ref{maps-on-cohomology} and
\ref{maps-on-cohomology2} that
$$H^\bullet(\Phi _{\cO _{\Gamma (g)}})\cdot
H^\bullet(\Phi _{\cO _{\Gamma (f)}})(-)=
\sqrt{td_X}\cup g^*(td_Y^{-1}\cup f_*(\sqrt{td_X}\cup -)).$$
Notice that the maps $f_*$ and $g^*$ preserve the degree of
cohomology and the Todd class $td$ is a power series in Chern classes
with constant term 1. It follows that the supertrace (and the trace)
of operators $H^\bullet(\Phi _{\cO _{\Gamma (g)}})\cdot
H^\bullet(\Phi _{\cO _{\Gamma (f)}})$ and $g^*\cdot f_*$ is the same.
This proves the proposition and the theorem.
\end{proof}
\end{proof}

Note that in the above proof of Theorem \ref{lefschetz-theorem-for-two-maps}
the assumption $\dim X=\dim Y$ was used only at the very end when we
derived Proposition \ref{hochschild-to-singular-for-two-maps}
from Lemmas \ref{maps-on-cohomology} and \ref{maps-on-cohomology2}.
(Without this assumption Theorem \ref{lefschetz-theorem-for-two-maps}
is false: take for example $X=\bbP^1$ and $Y=\pt.$)

In the above notation consider the composition of maps
$$H^\bullet(\Phi _{\cO _{\Gamma (f)}})
\cdot H^\bullet (\Phi _{\cO _{\Gamma (g)}})=
H^\bullet(\Phi _{\cO _{\Gamma (f)}*\cO _{\Gamma (g)}}).$$
This composition preserves the parity of cohomology and so is the
sum of two maps $H^{\even}(\Phi _{\cO _{\Gamma (f)}*\cO _{\Gamma (g)}})$
and $H^{\odd}(\Phi _{\cO _{\Gamma (f)}*\cO _{\Gamma (g)}}).$
Then our proof of Theorem \ref{lefschetz-theorem-for-two-maps} also gives
the following

\begin{theo}
 Let $X$ and $Y$ be smooth complex projective varieties
and let $f,g:X\to Y$ be two regular maps.
 Then in the previous notation there is the equality
$$\cO _{\Gamma (f)}\cdot \cO _{\Gamma (g)}=
\Tr H^{\even}(\Phi _{\cO _{\Gamma (f)}*\cO _{\Gamma (g)}})-
\Tr H^{\odd}(\Phi _{\cO _{\Gamma (f)}*\cO _{\Gamma (g)}}).$$
\end{theo}

\part{Lefschetz fixed point theorem for smooth and proper DG algebras}

\section{Hochschild homology of DG categories}

Fix a ground field $k.$ All algebras and categories are assumed
to be $k$-linear. We write $\otimes $ for $\otimes _k$ unless specified
otherwise. We follow consistently the universal sign rule: if $x,y$ are
homogeneous elements, then $xy=(-1)^{\deg (x)\deg (y)}yx.$

For a general discussion of DG algebras, DG modules, DG categories,
etc. the reader may consult for example
\cite{BL},\cite{Ke1},\cite{Dr}. For us a {\it DG module} means a
{\it right} DG module.

\subsection{Hochschild homology}

Let us recall the Hochschild complex and Hochschild homology of DG
algebras and small DG categories \cite{Ke2},\cite{Shk}.

Let $A=(A,d)$ be a DG algebra. As usual the suspension $sA=A[1]$
denotes the shift of grading: for $a\in A$ we have
$\deg(sa)=\deg(a)-1.$ Consider the graded $k$-module
$$C_\bullet (A)=A\otimes T(A[1])=\bigoplus _{n=0}^\infty A\otimes
A[1]^{\otimes n}.$$ Its element $a_0\otimes sa_1\otimes sa_2\otimes
...\otimes sa_n$ is traditionally denoted by $a_0[a_1\vert a_2\vert
...\vert a_n].$ The space $C_\bullet (A)$ is equipped with the
differential $b=b_0+b_1,$ where $b_0$ and $b_1$ are anti-commuting
differentials defined by $$b_0(a_0)=da_0,\quad b_1(a_0)=0,$$ and
$$b_0(a_0[a_1\vert ...\vert a_n])=da_0[a_1\vert ...\vert a_n]-\sum
_{i=1}^n(-1)^{\eta _{i-1}}a_0[a_1\vert ...\vert da_i \vert
...\vert a_n],$$
$$
b_1(a_0[a_1\vert ...\vert a_n])=(-1)^{\deg (a_0)}a_0a_1[a_2\vert
...\vert a_n] +\sum _{1}^{n-1}(-1)^{\eta _i}a_0[a_1\vert
...\vert a_ia_{i+1} \vert ...\vert a_n]$$ $$ -(-1)^{\eta
_{n-1}(\deg (a_n)+1)}a_na_0[a_1\vert ...\vert a_{n-1}].
$$
for $n\neq 0,$ where $\eta _i=\deg (a_0)+\deg (sa_1)+...+\deg (sa_i).$
The complex $C_\bullet (A)$ is called the Hochschild
chain complex of $A$, and the Hochschild homology is defined as
$$HH _n(A)=H^{-n}(C_\bullet (A)).$$

Similarly one defines the Hochschild chain complex $C_\bullet (\cA)$
and the Hochschild homology $HH(\cA)=HH _\bullet (\cA)$ for a small DG
category $\cA.$ Namely, denote by $\cA ^{n+1}$ the set of sequences
of objects $\{X_0,X_1,...,X_n\},$ $X_i\in \cA.$ Fox a fixed
$\mathbb{X} =\{X_0,...,X_n\}$ denote by $C_\bullet (\cA
,\mathbb{X})$ the graded space $\Hom (X_n,X_0)\otimes \Hom
(X_{n_1},X_n)[1] \otimes ...\otimes \Hom (X_0,X_1)[1].$ Now equip
the space
$$C_\bullet (\cA)=\bigoplus _{n\geq 0}\bigoplus _{\mathbb{X} \in \cA
^{n+1}}C_\bullet (\cA ,\mathbb{X})$$ with the differential
$b=b_0+b_1$ defined in analogy with the above case of a DG algebra.
The complex $C_\bullet (\cA)$ is the Hochschild chain complex of the
DG category $\cA$ and
$$HH_n (\cA)=H^{-n}(C_\bullet (\cA))$$
is the Hochschild homology of $\cA.$

Clearly, a DG functor $F:\cA \to \cB$ between DG categories $\cA$
and $\cB$ induces a morphism of complexes $C(F):C_\bullet (\cA)\to
C_\bullet (\cB)$ and hence a morphism
$$HH(F):HH(\cA)\to HH(\cB).$$ The following fact
is proved in \cite{Ke2}.

\begin{prop} \label{invariance-HH}
Homotopy equivalent DG functors induce the same map on $HH.$
\end{prop}

Given a DG algebra $A$ we denote by $A\text{-mod}$ the DG category of
(right) DG $A$-modules. Let $\cP (A)\subset A\text{-mod}$ be the full DG
subcategory of h-projective DG modules. Then the (triangulated) homotopy category
$Ho(\cP (A))$ is equivalent to the derived category $D(A).$ Let $\Perf A\subset \cP(A)$
be the full DG subcategory of perfect DG modules, and $A^{\text{pre-tr}}\subset \Perf A$
be the pre-triangulated envelop of the DG $A$-module $A.$ Then by definition
$Ho(\Perf A)$ is the Karoubian closure of $Ho(A^{\text{pre-tr}}).$
Here is another result from \cite{Ke2}.

\begin{prop} The natural DG embeddings $A\to A^{\text{pre-tr}}\to \Perf A$ induce
isomorphisms $$HH(A)=HH(A^{\text{pre-tr}})=HH(\Perf A).$$
\end{prop}

\subsection{Kunneth isomorphism} Let $A$ be a DG algebra. Let us
recall the definition of the shuffle product
$$\sh :C_\bullet (A)\otimes C_\bullet (A)\to C_\bullet (A).$$
For $a_0[a_1\vert ...\vert a_n], b_0[b_1\vert ...\vert b_m]\in
C_\bullet (A)$ put
$$\sh (a_0[a_1\vert ...\vert a_n]\otimes b_0[b_1\vert ...\vert
b_m])=(-1)^\heartsuit a_0b_0\sh _{nm}[a_1\vert ...\vert a_n\vert
b_1\vert ...\vert b_m]$$ Here $\heartsuit =\deg (b_0)(\deg
(sa_1)+...+\deg (sa_n))$ and
$$\sh _{nm}[x_1\vert ...\vert x_n\vert x_{n+1}\vert ...\vert
x_{n+m}]=\sum _{\sigma}\pm [x_{\sigma ^{-1}(1)}\vert ...\vert
x_{\sigma ^{-1}(n)}\vert x_{\sigma ^{-1}(n+1)}\vert ...\vert
x_{\sigma ^{-1}(n+m)}]$$ where the sum is taken over all
permutations that don't shuffle the first $n$ and the last $m$
elements and the sign is computed using the usual rule
$xy=(-1)^{\deg (x)+\deg (y)}yx.$

Obviously, the shuffle product is functorial with respect to
morphisms of DG algebras.

If $B$ is another DG algebra, the natural homomorphisms of DG
algebras $A\to A\otimes B,$ $B\to A\otimes B$ induces morphisms of
complexes $C_\bullet(A)\to C_\bullet (A\otimes B),$ $C_\bullet(B)\to
C_\bullet (A\otimes B).$

\begin{theo} The composition $K$ of maps
$$C_\bullet(A)\otimes C_\bullet (B)\to C_\bullet (A\otimes B)\otimes
C_\bullet(A\otimes B) \stackrel{\sh}{\to }C_\bullet (A\otimes B)$$
is a morphism of complexes which is a quasi-isomorphism.
\end{theo}

The Kunneth morphism $K$ which is defined in the previous theorem
for two DG algebras admits a generalization to the case of small DG
categories \cite{Shk},2.4, i.e. for DG categories $\cA, \cB$ we get
a functorial morphism
$$K:C_\bullet (\cA)\otimes C_\bullet (\cB)\to C_\bullet (\cA\otimes
\cB).$$

Let $A$ and $B$ be DG algebras. The obvious DG functor
$$\Perf A\otimes \Perf B\to \Perf (A\otimes B)$$ induces
a morphism of complexes
$$C_\bullet (\Perf A)\otimes C_\bullet (\Perf B)\to C_\bullet
(\Perf (A\otimes B)).$$ We denote the composition
$$C_\bullet (\Perf A)\otimes C_\bullet (\Perf B)\stackrel{K}{\to} C_\bullet
(\Perf A\otimes \Perf B)\to C_\bullet (\Perf (A\otimes B))$$ again
by $K.$

The next four lemmas are taken from
\cite{Shk},Prop.2.9,2.10,2.11,3.6.

\begin{lemma}
This map $K$ is a quasi-isomorphism.
\end{lemma}

\begin{proof}
Indeed, consider the commutative diagram of complexes
$$\begin{array}{ccc}
C_\bullet (A)\otimes C_\bullet (B) & \longrightarrow & C_\bullet
(A\otimes B)\\
\downarrow & & \downarrow \\
C_\bullet (\Perf A)\otimes C_\bullet (\Perf B) & \longrightarrow &
C_\bullet (\Perf (A\otimes B))
\end{array}
$$
The bottom arrow is a quasi-isomorphism because the other three are.
\end{proof}

\begin{lemma} Let $A,B,C$ be DG algebras. The diagram
$$\begin{array}{ccc}
C_\bullet (\Perf A)\otimes C_\bullet (\Perf B)\otimes C_\bullet
(\Perf C) & \stackrel {K\otimes 1}{\longrightarrow} &
C_\bullet(\Perf (A\otimes B))\otimes C_\bullet (\Perf C)\\
1\otimes K\downarrow & & \downarrow K\\
C_\bullet (\Perf A)\otimes C_\bullet (\Perf (B\otimes C))&
\stackrel{K}{\longrightarrow} & C_\bullet (\Perf (A\otimes B\otimes
C))\end{array}
$$
commutes. That is, the map $K$ is associative.
\end{lemma}

Let $A,B,C,D$ be DG algebras. Let $X\in A ^{\op}\otimes
C\text{-mod}$ and $Y\in B^{\op}\otimes D\text{-mod}$ be bimodules
which define functors
$$\Phi _X=-\otimes _AX:\Perf A\to \Perf C,\quad
\Phi _Y=-\otimes _BY:\Perf B \to \Perf D.$$

\begin{lemma} The following diagram
$$\begin{array}{lcr}
C_\bullet (\Perf A)\otimes C_\bullet (\Perf B) &
\stackrel{K}{\longrightarrow} & C_\bullet (\Perf (A\otimes B))\\
C(\Phi _X)\otimes C(\Phi _Y)\downarrow & & \downarrow C(\Phi _{X\otimes _kY})\\
C_\bullet (\Perf C)\otimes C_\bullet (\Perf D) &
\stackrel{K}{\longrightarrow} & C_\bullet (\Perf (C\otimes D)
\end{array}
$$
commutes.
\end{lemma}

\begin{lemma} For any DG algebra $A$ the formula
$$(a_0[a_1\vert ...\vert a_n])^\clubsuit =(-1)^{n+\sum_{1\leq i<j\leq
n}\deg (sa_i)\deg (sa_j)}a_0[a_n\vert a_{n-1}\vert ...\vert
a_1]$$ defines a quasi-isomorphism ${}^\clubsuit :C_\bullet (A)\to
C_\bullet (A^{\op}).$ A similar formula defines a quasi-isomorphism
${}^\clubsuit :C_\bullet (\cA)\to C_\bullet (\cA ^{\op})$ for any
small DG category $\cA.$ Clearly this quasi-isomorphism is preserved
by DG functors, i.e. given a DG functor $F:\cA \to \cB$ we have
${}^\clubsuit \circ C(F) =C(F^{\op})\circ {}^\clubsuit.$ So we obtain
a functorial isomorphism
$${}^\clubsuit :HH (\cA )\stackrel{\sim}{\longrightarrow}
HH(\cA ^{\op}).$$
\end{lemma}

\subsection{Euler class} Let $A$ be a DG algebra.
Recall the definition of the Euler class map
$\Perf A\to HH_0(A).$ Given $N\in \Perf A$ we consider the corresponding
functor $\Phi _N=-\otimes _kN:\Perf k\to \Perf A$ and define
$$Eu(N):=HH(\Phi _N)(1)\in HH_0(\Perf A)=HH_0(A).$$

Thus if $B$ is another DG algebra and $F:\Perf A\to \Perf B$ is a
DG functor then by definition
$HH(F)(Eu(N))=Eu(F(N)).$

The next two lemmas are Prop.3.1,3.2 in \cite{Shk}.

\begin{lemma} If $N,M\in \Perf A$ are homotopy equivalent then
$Eu(N)=Eu(M).$
\end{lemma}

\begin{lemma} For any exact triangle $L\to M\to N\to L[1]$ in
$Ho(\Perf A)$ we have $$Eu(M)=Eu(L)+Eu(N).$$ In particular
$Eu(N[1])=-Eu(N).$
\end{lemma}

\begin{cor} The map $Eu$ descends to a group homomorphism
$$Eu :K_0(Ho(\Perf A))\to HH_0(A)$$
\end{cor}

\begin{cor} \label{Eu-k-mod}
Let $N\in \Perf k.$ Then $\Eu (N)=\sum _i(-1)^i\dim H^i(N)\in
k=HH_0(k)=HH_\bullet (k).$
\end{cor}

\subsection{Pairing on $HH$}
Let $A$ be a DG algebra.
Consider $A$ as a left DG
$A$-bimodule, i.e. as a DG $A\otimes A^{\op}$-module via
$$(a\otimes b)c=(-1)^{\deg (b)\deg (c)}acb.$$ We denote by
$\Delta $ this left DG $A$-bimodule.

\begin{defi} \label{def-pairing}
Consider the DG functor $\Phi _{\Delta }=-\otimes _{A\otimes
A^{\op}}A:\Perf (A\otimes A^{\op})\to \Perf (k).$
The composition of maps
$$C_\bullet (\Perf A)\otimes C_\bullet (\Perf
A^{\op})\stackrel{K}{\longrightarrow} C_\bullet (\Perf (A\otimes
A^{\op}))\stackrel{C(\Phi _{\Delta })}{\longrightarrow} C_\bullet
(\Perf k)$$ defines the pairing
$$\langle ,\rangle =\langle ,\rangle _A :
HH(\Perf A)\otimes HH(\Perf A^{\op})\to k.$$ Using the canonical
isomorphism $HH (A)=HH(\Perf A)$
we also get the pairing
$$\langle ,\rangle :HH (A)\otimes HH (A^{\op})\to k.$$
\end{defi}

We can apply the previous construction to the DG algebra $A^{\op}$
instead of $A$ to get the pairing
$$\langle ,\rangle =\langle ,\rangle _{A^{\op}}:HH(A^{\op})\otimes HH(A)\to k.$$

\begin{lemma} \label{compare-pairings}
For $x\in HH (A),y\in HH (A^{\op})$
we have
$$\langle x,y\rangle _A =
(-1)^{\deg (x)\deg (y)}\langle y,x\rangle _{A^{\op}} .$$
\end{lemma}

\begin{proof} Denote by $\Delta ^{\op}$ the left
DG $A^{\op}\otimes A$-module $A$ via the action
$$(a\otimes b)c=(-1)^{\deg (a)(\deg (b) +\deg (c))}bca.$$
Then by definition the pairing $\langle ,\rangle _{A^{\op}}$
is defined by the composition of maps
$$C_\bullet (\Perf A^{\op})\otimes C_\bullet (\Perf
A)\stackrel{K}{\longrightarrow} C_\bullet (\Perf (A ^{\op}\otimes
A))\stackrel{C(\Phi _{\Delta ^{\op} })}{\longrightarrow} C_\bullet
(\Perf k)$$

Note that the isomorphism of DG algebras
$$\sigma :A\otimes A^{\op} \to A^{\op}\otimes A,\quad \sigma (a\otimes b)=
(-1)^{\deg (a)\deg (b)}b\otimes a$$
interchanges the left DG modules
$\Delta $ and $\Delta ^{\op}.$ Thus we obtain the induced commutative
diagram of DG functors
$$\begin{array}{ccccc}
 A \otimes  A^{\op} & \longrightarrow & \Perf (A\otimes A^{\op}) &
\stackrel{\Phi _{\Delta }}{\longrightarrow } & \Perf k\\
\sigma \downarrow & & \sigma \downarrow & & \vert \vert \\
 A ^{\op} \otimes  A & \longrightarrow &
\Perf (A^{\op}\otimes A)  &
\stackrel{\Phi _{\Delta ^{\op} }}{\longrightarrow } & \Perf k
\end{array}
$$
Consider the diagram
$$\begin{array}{ccc}
C_\bullet (A)\otimes C_\bullet (A^{\op}) & \stackrel{K}{\longrightarrow} &
C_\bullet (A\otimes A^{\op})\\
 & & \downarrow C_\bullet (\sigma )\\
C_\bullet (A^{\op})\otimes C_\bullet (A) & \stackrel{K}{\longrightarrow} &
C_\bullet (A^{\op}\otimes A)
\end{array}
$$
It remains to notice that the induced isomorphism
$$K^{-1}\cdot HH(\sigma )\cdot K :HH (A)\otimes HH ( A^{\op})\to
HH ( A^{\op})\otimes HH ( A)$$
maps $x\otimes y$ to $(-1)^{\deg (x)\deg (y)}y\otimes x.$
This implies the lemma.
\end{proof}

\begin{defi} \label{def-HH(M)}
Recall that for $M\in \Perf (A\otimes A^{\op})$ (resp.
$M\in \Perf (A^{\op}\otimes A)$) the group
$HH_i(M):=H^{-i}(\Phi _{\Delta}(M))$ (resp.
$HH_i(M):=H^{-i}(\Phi _{\Delta ^{\op}}(M))$)
is called the $i$-th Hochschild
homology group of $M.$
\end{defi}

\subsection{Smooth and proper DG algebras}\label{smooth-proper-algebras}
 Recall that a DG algebra $A$ is {\it smooth} if it is perfect as a DG
$A^{\op}\otimes A$-module. It is called {\it proper} if its total
cohomology $H(A)$ is finite dimensional.

\begin{lemma} Let $A$ and $B$ be smooth and proper DG algebras.

(a) The DG algebras $A^{\op}$ and $A\otimes B$ are also smooth and
proper.

(b) A DG $A$-module $N$ is perfect if and only if its total
cohomology $H(N)$ is finite dimensional.

(c) Any DG module $L\in \Perf (A^{\op}\otimes B)$ defines the functor
$$\Phi _L:\Perf A\to \Perf B.$$

(d) The total Hochschild homology $HH(A)$ is finite
dimensional.

(e) For $M\in \Perf (A^{\op}\otimes A)$ the total Hochschild
homology $HH(M)$ is finite dimensional.
\end{lemma}

\begin{proof} (a) See for example \cite{Lu}.

(b) Since $A$ is proper it is clear that a perfect DG $A$-module has
finite dimensional cohomology.

Vice versa, assume that $N$ has finite dimensional cohomology. Since
the DG algebra $A$ is smooth, the DG $A^{\op}\otimes A$-module $A$
is a homotopy direct summand of a DG $A^{\op}\otimes A$-module $P$
which is obtained from $A^{\op}\otimes A$ by an iterated cone
construction. Then $N\simeq N\stackrel{\bL}{\otimes }_AA$ is a
homotopy direct summand of $N\otimes _AP,$ which is obtained from
the DG $A$-module $N\otimes _kA\simeq H(N)\otimes _kA$ by an
iterated cone construction.
Since $H(N)$ has finite dimension it follows that $N\in \Perf A.$

(c) This follows from (b).

(d) Since $HH_\bullet
(A)=H^{-\bullet}(A\stackrel{\bL}{\otimes}_{A^{\op}\otimes A}A)$ the
statement is clear.

(e) This follows because
$HH_\bullet(M)=H^{-\bullet}(M\stackrel{\bL}{\otimes}_{A^{\op}\otimes
A}A).$
\end{proof}

\section{Main theorems}

\subsection{Formulation of main theorems}

\begin{theo} \label{nondegenerate} (\cite{Shk},Thm.6.2)
Let $A$ be a smooth and proper DG algebra. Then the
pairing $\langle ,\rangle _A$ is nondegenerate.
\end{theo}

\begin{theo} \label{Lef} ({\rm LFP})
Let $A$ be a smooth and proper DG algebra and
$M\in \Perf
(A^{\op}\otimes A).$ Consider the functor $\Phi _M=-\stackrel
{\bL}{\otimes} _AM:
\Perf A\to \Perf A$ and the corresponding linear endomorphisms
$HH_j(\Phi _M):HH_j(A)\to HH_j(A).$ Then there is an equality of the
two elements of $k$
$$\sum _i(-1)^i\dim HH_i(M)=\sum _j(-1)^j\Tr HH_j(\Phi _M))$$
\end{theo}

\begin{theo} \label{HRR} ({\rm HRR} \cite{Shk},Thm.3.5)
Let $A$ be a proper DG algebra.
 For any $N\in
\Perf A,$ $M\in \Perf A^{\op}$
$$\sum _i\dim H^i(N\stackrel{\bL}{\otimes } _AM)=
\langle Eu(N),Eu(M)\rangle.$$
\end{theo}

\begin{remark} \label{generalization-by-petit} In the recent paper \cite{Pe}
there appears a generalization of Theorem \ref{HRR} using the Euler class
of a pair $(M,f),$ where $M\in \Perf A$ and $f:M\to M$ is an endomorphism.
\end{remark}

\subsection{Proofs of main Theorems}

Everything is a consequence of the
following key lemma proved in \cite{Shk},Thm.3.4.

For $X\in \Perf (A^{\op}\otimes B)$ denote by $Eu ^\prime (X)$ the
element
$$K^{-1}(Eu(X))\in \bigoplus _nHH_{-n}(\Perf A^{\op})\otimes
HH_n(\Perf B),$$ where $K$ is the Kunneth isomorphism.

Note that if the DG algebra $A$ is proper, then the functor
$\Phi _X=-\otimes _AX$ maps $\Perf A$ to $\Perf B.$

\begin{lemma} \label{main-lemma}
Let $A,B$ be DG algebras and $X\in
\Perf(A^{\op}\otimes B).$ Assume that $A$ is proper. Then the map
$$HH(\Phi _X):HH(A)\to HH(B)$$ is the convolution with the class $Eu
^\prime (X).$ That is, if
$$Eu ^\prime (X)=\sum _nx^\prime _{-n}\otimes x_n \in
\bigoplus _nHH_{-n}(\Perf A^{\op})\otimes HH_n(\Perf B),$$ then
$HH(T_X)(y)=\sum _n\langle y,x^\prime _{-n}\rangle \cdot x_n.$
\end{lemma}

\begin{proof} Note that the DG functor $\Phi _X $ is isomorphic
to the composition of DG functors
$$\Perf A\stackrel{-\otimes _kX}{\longrightarrow} \Perf (A\otimes
A^{\op}\otimes B) \stackrel{\Phi _{\Delta \otimes _kB}}{\longrightarrow}
\Perf B$$
It follows that the corresponding map $HH(\Phi _X)$ is isomorphic to
the following composition
$$\begin{array}{ccc}
HH(\Perf A)\otimes HH(\Perf k) & & HH(\Perf B)\\
\downarrow & & \uparrow \\
HH(\Perf A) & & HH(\Perf k)\otimes HH(\Perf B) \\
HH(-\otimes _kX)\downarrow & & \uparrow HH(\Phi _{\Delta})\otimes \id \\
HH(\Perf (A\otimes A^{\op}\otimes B)) & & HH(\Perf (A\otimes A^{\op}))
\otimes HH(\Perf B)\\
K^{-1}\downarrow & & \uparrow K\otimes \id \\
HH(\Perf A)\otimes HH(A^{\op}\otimes B) & \stackrel{\id \otimes K^{-1}}
{\longrightarrow} & HH(\Perf A)\otimes HH(\Perf A^{\op}) \otimes
HH(\Perf B)
\end{array}
$$
The composition of the left vertical arrows equals
$$ HH(\Perf A)\otimes HH(\Perf k)\stackrel{\id \otimes Eu^\prime (X)}
{\longrightarrow} HH(\Perf A)\otimes HH(\Perf (A^{\op}\otimes B)).$$
This implies the lemma.
\end{proof}

\subsection{Proof of Theorem \ref{HRR}}
We apply Lemma \ref{main-lemma} with $A=A, B=k$ and $X=M.$ The
composition of DG functors
$$\Perf k\stackrel{\Phi _N}{\longrightarrow} \Perf A \stackrel{\Phi _M}
{\longrightarrow} \Perf k$$ is isomorphic to the DG functor
$\Phi _{N\otimes _AM}:\Perf k\to \Perf k.$ By Lemma \ref{main-lemma}
$$Eu (N\otimes _AM)=\langle Eu(N),Eu(M)\rangle$$
and by Corollary \ref{Eu-k-mod} $Eu(N\otimes _AM)=\sum _i(-1)^i\dim
H^i(N\otimes _AM).$

\subsection{Proof of Theorem \ref{nondegenerate}} We apply Lemma
\ref{main-lemma} with $A=B=X.$ Then the
functor $\Phi _X:\Perf A\to \Perf A$ is isomorphic to the identity.
Therefore the corresponding linear map $HH (\Phi _X):
HH (A)\to HH (A)$ is the identity. By Lemma
\ref{main-lemma} this shows that the map $HH (A)\to
HH(A^{\op})^*$ defined by the pairing $\langle ,\rangle _A$
is injective. Since the space $HH(A)$ is finite dimensional and
is isomorphic to $HH (A^{\op})$ it follows that the pairing is
nondegenerate.

\subsection{Proof of Theorem \ref{Lef}} Fix $M\in \Perf (A^{\op}\otimes A).$
As before
denote by $Eu (M)^\prime \in HH (A^{\op})\otimes  HH (A)$
the inverse
image of $Eu(M)$ under the Kunneth isomorphism
$$HH (A^{\op})\otimes HH (A)\stackrel{K}{\longrightarrow}
HH(A^{\op}\otimes A).$$

Choose a homogeneous basis $\{v_m\}$ of $HH (A)$ and let
$\{\bar{v}_m\}$ be a basis of
$HH(A^{\op})$ such that
$\langle \bar{v}_m,v_n\rangle _{A^{\op}}=\delta _{mn}$ (we use Theorem
\ref{nondegenerate}).
Let
$$Eu (M)^\prime =\sum _{m,n}\alpha _{mn}\cdot \bar{v}_m\otimes v_n$$
for $\alpha _{mn}\in k.$
Then by Definitions \ref{def-pairing}, \ref{def-HH(M)} and Corollary
\ref{Eu-k-mod}
$$\sum _i(-1)^i \dim HH_i(M)=\sum _{m}\alpha _{mm}.$$
On the other hand by Lemma \ref{main-lemma}
$$HH (\Phi _M)(v_l)=
\sum _{m,n}\alpha _{mn}\cdot \langle v_l,\bar{v}_m\rangle _A\cdot v_n$$
By Lemma \ref{compare-pairings}
$\langle v_l,\bar{v}_m\rangle _A=(-1)^{\deg (v_l)\deg (\bar{v}_m)}
\langle \bar{v}_m,v_l\rangle _{A^{\op}}=(-1)^{\deg (v_l)}
\langle \bar{v}_m,v_l\rangle _{A^{\op}}=(-1)^{\deg (v_l)}\delta _{lm}.$
So the trace of the linear operator
$HH (\Phi _M)$ on $HH (A)$ equals $\sum _m(-1)^{\deg (v_m)}\alpha _{mm}.$
Hence its supertrace is
$$\sum _i(-1)^i\Tr HH_i(\Phi _M)=\sum _m\alpha _{mm}$$
which proves the theorem.

\end{document}